\documentclass{article}
\usepackage[final]{neurIPS_2019_modified}

\usepackage{booktabs} 
\usepackage{graphicx} 
\usepackage{subcaption}
\usepackage[utf8]{inputenc} 
\usepackage[T1]{fontenc}    
\usepackage{url}            
\usepackage{booktabs}       
\usepackage{amsfonts}       
\usepackage{nicefrac}       
\usepackage{microtype}      
\usepackage{amsthm}
\usepackage{algorithm,algpseudocode}
\RequirePackage{amsmath, subcaption}
\usepackage{amssymb}
\usepackage{ulem}
\usepackage{graphicx}
\usepackage{caption, geometry}
\usepackage{enumitem}
\usepackage[font=footnotesize,labelfont=bf]{caption}

\newtheorem{assumption}{Assumption}
\newtheorem{lemma}{Lemma}

\newtheorem{theorem}{Theorem}
\theoremstyle{remark}

\newtheorem{definition}{Definition}

\input{Definitions}

\newcommand{\bxi}{{\boldsymbol{\xi}}}
\newcommand{\btxi}{{\boldsymbol{\tilde{\xi}}}}
\newcommand{\real}{{\mathbb R}}
\newcommand{\Ebb}{{\mathbb E}}
\title{Optimal Convergence for Stochastic Optimization with Multiple Expectation Constraints}

\author{Kinjal Basu \\
 LinkedIn Corporation\\
 Mountain View, CA 94083\\
 \texttt{kbasu@linkedin.com}\\
 \And Preetam Nandy \\
 LinkedIn Corporation\\
 Mountain View, CA 94083\\
 \texttt{pnandy@linkedin.com}\\
 }

\begin{document}
\maketitle

\begin{abstract}
In this paper, we focus on the problem of stochastic optimization where the objective function can be written as an expectation function over a closed convex set. We also consider multiple expectation constraints which restrict the domain of the problem. We extend the cooperative stochastic approximation algorithm from \cite{lan2016algorithms} to solve the particular problem. We close the gaps in the previous analysis and provide a novel proof technique to show that our algorithm attains the optimal $\Ocal(1/ \sqrt{N})$ rate of convergence for both optimality gap and constraint violation when the functions are generally convex. We also compare our algorithm empirically to the state-of-the-art and show improved convergence in many situations. 
\end{abstract}

\vspace{-0.2cm}
\section{Introduction}
\label{sec:introduction}
\vspace{-0.3cm}
In this paper we focus on a stochastic optimization problem with multiple expectation constraints. Specifically, we are interested in solving a problem of the following form:
\begin{equation}
\label{eq:main_problem}
\begin{aligned}
& \underset{x}{\text{min}} & & f(x) := \Ebb_{ \xi_0} \left(F(x, \xi_0)\right)\\
& \text{subject to}
& & g_j(x) :=  \Ebb_{ \xi_j} \left(G_j(x, \xi_j)\right) \leq 0\qquad \text{ for } j = 1, \ldots, m. \\
&&& x \in \Xcal
\end{aligned}
\end{equation}
where $\Xcal \subseteq \real^n$ is a convex compact set, $\xi_j$ are random variables 
for $j = 0, \ldots, m$. 
$F(\cdot), G_j(\cdot)$ for $j = 1, \ldots, m$ are closed convex functions with respect to $x$ for a.e. $\xi_j \in \Pcal_j$ for $j = 0, \ldots, m$. 

There are several applications of the above problem formulation \eqref{eq:main_problem} especially, in fields such as control theory \cite{primbs2009stochastic}, management science \cite{charnes1959chance}, finance \cite{rockafellar2000optimization}, etc.
Our specific motivation comes from a problem arising from a large-scale social network platform. Any such platform does extensive experimentation to identify which parameters should be applied to certain members to get the most amount of metric gains. An example of such an influential parameter is the gap between successive ads on a newsfeed product (described in details in Section \ref{sec:experiments}). 

Traditionally, stochastic optimization routines were solved either via sample average approximation (SAA) \cite{kleywegt2002sample, shapiro2003monte, wang2008sample} or via stochastic approximation (SA) \cite{robbins1951stochastic}.  In SAA, each function $g_j$ for $j = 1, \ldots, m$ are approximated by a sample average $\hat{g}_j(x) = \sum_{i=1}^n G_j(x, \xi_{j,i}) /n$ and then solved via traditional means. The SAA solution is computationally challenging, may not be applicable to the online setting and the approximation might lead to an infeasible problem. In SA, the algorithm follows the usual gradient descent algorithm by using the stochastic gradient rather than $f'(x)$ to solve \eqref{eq:main_problem} \cite{benveniste2012adaptive, spall2005introduction}. The SA solution requires a projection onto the domain specified by $\{g_j(x) \leq 0\}$, which may not be possible when we have the expectation formulation. 

There have been several papers showing improvement over the original SA method especially for strongly convex problems \cite{polyak1992acceleration}, and for general class of non-smooth convex stochastic programming problems \cite{nemirovski2009robust}. However, these methods may not be directly applicable to the case where each $g_j$ is an expectation constraint. Very recently \cite{yu2017online} developed a method of stochastic online optimization with general stochastic constraints by generalizing Zinkevich's online convex optimization. Their approach is much more general and as a result may be not the optimal approach to solving this specific problem. For more related works, we refer to \cite{lan2016algorithms} and \cite{yu2017online} and the references therein.

Lan et. al. (2016) \cite{lan2016algorithms} introduced the cooperative stochastic approximation (CSA) algorithm to solve problem \eqref{eq:main_problem} with $m=1$. In this paper, we extend their algorithm to multiple expectation constraints and also close several gaps in their proof of optimal convergence. Specifically, we prove that an optimal point $\hat{\xb}$ satisfying 
\begin{align*}
\Ebb(f(\hat{\xb}) - f(\xb^*)) \leq c/\sqrt{N} \text{ and }  \Ebb(g_j(\hat{\xb}))  \leq C/\sqrt{N}  \;\; \forall j \in \{1,\ldots, m\}
\end{align*}
can be obtained with $N$ steps of our algorithm. These rates are optimal due to the lower bound results following from \cite{agarwal2009information}. We primarily focus on the theoretical analysis of the algorithm in this paper. We use a completely novel proof technique and overcome several gaps in the proof of the \cite{lan2016algorithms} which we highlight in Section \ref{sec:convergence}. Furthermore, we run experiments on simulated data to show that this algorithm empirically outperforms the algorithm in \cite{yu2017online}. For more practical results, we refer the reader to \cite{tu2019prophet}.

The rest of the paper is organized as follows. In Section \ref{sec:algo}, We introduce the multiple cooperative stochastic approximation (MCSA) algorithm for solving the problem stated in \eqref{eq:main_problem}. In Section \ref{sec:convergence}, we prove the convergence guarantee of our algorithm. We discuss some empirical results in Section \ref{sec:experiments} before concluding with a discussion in Section \ref{sec:discussion}.

\vspace{-0.2cm}
\section{Multiple Cooperative Stochastic Approximation (MCSA)}
\label{sec:algo}
\vspace{-0.3cm}
\label{sec:mcsa}
%
We begin with the definition of a projection operator. Let $\psi : \Xcal \rightarrow \real$ be a 1-strongly convex proximal function, i.e., we have
$\psi(y) \geq \psi(x) + \langle \nabla \psi(x), y - x \rangle + \frac{1}{2} \norm{x-y}_{}^2.$
The Bregman divergence \cite{duchi2011adaptive} associated with a strongly convex and differentiable function $\psi$ is defined as
$B_{\psi}(x, y) := \psi(x) - \psi(y) - \langle \nabla \psi(y), x - y \rangle.$
Note that due to the strong convexity of $\psi$ we have $B_{\psi}(y, x) \geq \frac{1}{2} \norm{x-y}_{2}^2 \;\; \forall x,y  \in \Xcal. $ 

\begin{definition}\label{definition: projection}
Based on the function $\psi$ we define the proximal projection function $P_{x, \Xcal}^{\psi}: \real^n \rightarrow \real^n$ as
\vspace{-0.3cm}
\begin{align}
\label{eq:prox_proj}
P_{x, \Xcal}^{\psi}(\cdot) := \argmin_{z \in \Xcal} \left\{ \langle \cdot, z \rangle + B_{\psi}(z,x) \right\}.
\end{align}
We fix $\psi$ and in the rest of the paper, we denote the proximal projection function by $P_x(\cdot)$. 
\end{definition}

We assume that the objective function $f(x)$ and each constraint function $g_j(x)$ in \eqref{eq:main_problem} are well-defined, finite valued, continuous and convex functions for $x \in \Xcal$. Similar to the CSA algorithm \cite{lan2016algorithms}, at each iterate $t > 0$, we move along the subgradient direction $f'(x_t)$ if all $g_j(x_t) \leq \eta_{j,t}$ for all $j$ and for the tolerance sequence $\{\eta_j\}_t > 0$. Otherwise we move along a randomly chosen subgradient direction $g'_{j^*}(x_t)$ where $j^*$ is chosen randomly from the set $\left\{j : g_j(x_t) > \eta_{j,t} \right\}$. This modification to the CSA algorithm allows us to work with multiple constraints. At each stage, we move along the chosen direction with a stepsize $\gamma_t$. 
We will show in the Sections \ref{sec:main_result} how we can choose the tolerances $\{\eta_{j,t}\}, \{\gamma_t\}$ so that we can achieve the optimal convergence rates. 

Since we do not have access to the exact functions $f', g_j'$ and $g_j$ for $j = 1,\ldots, m$, we use an approximation $\hat{G}_j$ for $g_j$ and use the stochastic gradients $G_j'( x_t, \xi_{j,t})$ and $F'(x_t, \xi_{0,t})$ for $g_j'(x_t)$ and $f'(x_t)$ respectively, where $\xi_{j,t}$ are i.i.d. observations that we get to observe from the distribution of $\xi_j$ at iteration $t$. We run $N$ steps of our algorithm and choose our final solution $\hat{\xb}$ as the mean over a set of indices that is defined by
\begin{align}
\label{eq:Bdef}
\Bcal = \left\{ s \leq t \leq N  :  \hat{G}_{j}(x_t) \leq \eta_{j,t} \; \forall j \in \{1, \ldots, m\} \right\}
\end{align}
Here $s$ denotes a burn-in period.
Note that we need to get an approximation of $g_j(x_t)$ for every stage $t$ for all $j$. A consistent estimate of $g_j(x_t)$ is given by
\vspace{-0.2cm}
\begin{align}
\label{eq:Gfunc1}
\hat{G}_{j}(x_t) := \frac{1}{L} \sum_{\ell = 1}^L G_j(x_t, \tilde{\xi}_{j,\ell}),
\end{align}
where $\tilde{\xi}_{j,\ell}$ for $\ell = 1, \ldots, L$ are i.i.d.\ observations from the distribution of $\xi_j$. Throughout this paper we assume that $L$ grows in the same rate as $N$. The full algorithm is  written out as Algorithm \ref{algo:mcsa}.

\begin{algorithm}[!h]
\caption{Multiple Cooperative Stochastic Approximation}\label{algo:mcsa}
\begin{algorithmic}[1]
\State Input : Initial $x_1 \in \Xcal$, Tolerances $\{\eta_j\}_t, \{\gamma\}_t$, The number of iterations $N$
\For{ $t = 1, \ldots, N$ }
\State Estimate $\hat{G}_{j}(x_t)$ for all $j \in 1,\ldots, m$ using \eqref{eq:Gfunc1}.
\If{$\hat{G}_{j}(x_t) \leq \eta_{j,t} \text{ for all } j$}
\State Set $h_t = F'(x_t, \xi_{0,t})$
\Else
\State Randomly select $j^*$ from $\{j : \hat{G}_{j}(x_t) > \eta_{j,t}\}$
\State Set $h_t = G_{j^*}'(x_t, \xi_{j^*,t})$
\EndIf
\State Compute $x_{t+1} = P_{x_t}(\gamma_t h_t)$
\EndFor
\State Define $\Bcal = \left\{ s \leq t \leq N  : \hat{G}_{j}(x_t) \leq \eta_{j,t} \; \forall j \in \{1, \ldots, m\} \right\}$
\State \Return $\hat\xb := \frac{\sum_{t \in \Bcal} x_t \gamma_t}{\sum_{t \in \Bcal} \gamma_t}$
\end{algorithmic}
\end{algorithm}

\vspace{-0.5cm}
\section{Convergence Analysis}
\label{sec:convergence}
\vspace{-0.3cm}
In this section, we study the convergence of the MCSA algorithm as described in Section \ref{sec:mcsa}. Specifically, we show that if we run $N$ iterations in Algorithm \ref{algo:mcsa}, then expected error and expected violation of the constraints is $\Ocal(1/\sqrt{N})$. Note that the presence of interrelated stochastic components of the MCSA algorithm poses a novel challenge in deriving the rate of convergence of MCSA. In particular, at every step MCSA makes a random decision for choosing $h_t$ depending on $x_t$ and $\{\tilde{\xi}_{j, \ell} : j = 1, \ldots, m,~ \ell = 1,\ldots, L \}$, while $x_t$ is a random variable depends on $x_{t-1}$, $\{\xi_{j,t} : j = 0,\ldots, m \}$ and a similar random decision made at the previous step for generating $x_{t-1}$. A careful consideration of this intertwined stochasticity reveals several gaps in the convergence analysis of the CSA algorithm as presented in \cite{lan2016algorithms}. For this reason, we refrained from deriving the rate of convergence of MCSA by extending the convergence analysis of \cite{lan2016algorithms}. Our rigorous convergence analysis is the main contribution of this paper and this also reinforces the $\Ocal(1/\sqrt{N})$ convergence rate of the CSA algorithm \cite{lan2016algorithms}.



\textbf{Proof Overview}: We begin with the identification of a sufficient condition on the threshold parameters $\{\eta_{j,t}\}, \{\gamma_t\}$ such that the solution $\hat{\xb}$ is well-defined and on average consists of at least $cN$ $x_t$'s for some $c \in (0, 1]$ (Theorem \ref{theorem:welldefined}). In order to show that result, we use a concentration bound (Lemma \ref{lemma:conc}) as well as the properties of a Donsker class of functions (Lemma \ref{lemma:donsker}). Finally, we prove the $\Ocal(1/\sqrt{N})$ convergence rate of the MCSA algorithm in Theorem \ref{thm:main1} using Theorem \ref{theorem:welldefined}, an application of the FKG inequality (Lemma \ref{lemma:fkg application}) and Lemma \ref{lemma:donsker}. We use the FKG inequality to untangle the dependency between the random variables $\{x_t : t \in \Bcal\}$ and the random set $\Bcal$. 

\vspace{-0.2cm}
\subsection{Supporting Results} 
\label{sec:supporting}
\vspace{-0.2cm}
\textbf{Bregman Divergence:} We begin by stating our first lemma connecting Bregman divergence and the proximal projection function \eqref{eq:prox_proj}. All proofs of lemmas in this sub-section are pushed to the supplementary material for brevity.
\begin{lemma}
\label{lemma:breg}
For any $x,z \in \Xcal$ and $y \in \real^n$, we have $B_{\psi}\left( z, P_x(y) \right) \leq B_{\psi}( z, x ) + y^T(z - x) + \frac{1}{2}\norm{y}_{{\psi}^*}^2 $
where $\norm{\cdot}_{{\psi}^*}$ is the dual norm of $\norm{\cdot}_{\psi}$.
\end{lemma}
The following result follows from Lemma \ref{lemma:breg} and a careful analysis of Algorithm \ref{algo:mcsa}.
\begin{lemma}
\label{lemma:prop9}
For any $s \in T$, we have
\vspace{-0.2cm}
\begin{align}
\label{eq:prop9}
\sum_{t \in \Bcal}^N \gamma_t \langle F'(x_t, \xi_{0,t}), x_t - x \rangle &+ \sum_{j=1}^m \sum_{t \in \Ncal_j}\gamma_t \big(G_j(x_t, \xi_{j,t}) -  G_j(x, \xi_{j,t})\big) \leq B_{\psi}(x, x_{s}) \nonumber \\
&+ \sum_{t \in \Bcal} \frac{\gamma_t^2}{2} \norm{F'(x_t, \xi_{0,t})}_{\psi^*}^2 +  \sum_{j=1}^m \sum_{t \in \Ncal_j}\norm{G'(x_t, \xi_{j,t})}_{\psi^*}^2 \;\; \text{a.s.}
\end{align}
\end{lemma}
We bound the right hand side of \eqref{eq:prop9} in Lemma \ref{lemma:prop9} by making the following assumption.
\begin{assumption}
\label{assump:norm}
For any $x \in \Xcal$, the following holds
\begin{align*}
\Ebb( \norm{F'(x, \xi_0)}_{{\psi}^*}^2)  \leq M_F^2 \qquad \text{ and } \qquad
\Ebb( \norm{G'_j(x, \xi_j)}_{{\psi}^*}^2)  \leq M_{G_j}^2 \;\; \forall \; j \in \{1, \ldots, m\}.
\end{align*}
\end{assumption}
\begin{lemma}
\label{lemma:prop9b}
Under Assumption \ref{assump:norm}, for any $s \in T$, we have
\vspace{-0.2cm}
\begin{align}
\label{eq:prop9b}
~\sum_{t = s}^N \sum_{j=1}^m \gamma_t~\Ebb\bigg[\big(g_j(x_t) -  g_j(x) \big) ~\one_{\{t \in \Ncal_j\}} \biggm| \btxi \bigg] ~&+~ \sum_{t = s}^N  \gamma_t ~\Ebb \bigg[ \langle f'(x_t), x_t - x \rangle ~\one_{\{t \in \Bcal\}} \biggm| \btxi \bigg] \nonumber \\
& \leq ~ \Ebb \bigg[ B_{\psi}(x, x_{s}) \bigg| \btxi \bigg]  ~+~ \frac{M^2}{2}~ \sum_{t = s}^N\gamma_t^2,
\end{align}
where $M^2 := \max \{M_F^2, M_{G_1}^2, \ldots, M_{G_m}^2\}$, $\btxi :=  \{ \{ \tilde{\xi}_{j,\ell}\}_{\ell = 1}^L\}_{j=1}^m$ and $\Ncal_j \subseteq \{s, \ldots, N\}$ denotes the random set of indices for which $j^* = j$ in Algorithm \ref{algo:mcsa}.
\end{lemma}

\textbf{Concentration Bounds:} We use a concentration result to achieve the optimal convergence. Towards that, we first define, 
\vspace{-0.2cm}
\begin{align*}
\zeta_t = \sum_{j=1}^m \gamma_t \left( g_j(x_t)\one_{\{t \in \Ncal_j\}} -  \Ebb \left[ g_j(x_t) ~\one_{\{t \in \Ncal_j\}} \bigg| \btxi \right]\right).
\end{align*}
Now, we assume that the distribution of $\zeta_t$ has a light-tail.   
\begin{assumption}
\label{assump:lighttail}
There exists a $\sigma^2$ such that $\Ebb[ \zeta_t^2 / (\gamma_t^2 \sigma^2) ~ | ~ \btxi ] \leq  \exp(1).$
\end{assumption}
\begin{lemma}
\label{lemma:conc}
Under Assumption \ref{assump:lighttail}, for any $\lambda > 0$,
$P\big(\sum_{t=s}^{N} \zeta_t > \lambda \sqrt{\sum_{t=s}^N \gamma_t^2 \sigma^2} \;\big|\;\btxi \big) \leq \exp\big(-\frac{\lambda^2}{3}\big).$
\end{lemma}

\textbf{Donsker Class:} Recall that we approximate $g_j(x_t)$ by $\hat{G}_{j}(x_t)$ in Algorithm \ref{algo:mcsa}. Although the central limit theorem guarantees the convergence of $\sqrt{L}(\hat{G}_{j}(x) - g_j(x))$ to a zero mean Gaussian distribution for each $x$, it does not ensure a ``uniform convergence'' for all $x \in \Xcal$ as in Lemma \ref{lemma:donsker}. To achieve that we show $\Fcal_j = \{ G_j(x, \cdot): x \in \Xcal \}$ is a \emph{Donsker class} \cite{VaartWellner96} under the following assumption.


\begin{assumption}
\label{assump:donsker}
For each $j = 1, \ldots, m$, the class of functions $\Fcal_j = \{ G_j(x, \cdot): x \in \Xcal \}$ satisfies the following Lipschitz condition:
\vspace{-0.2cm}
\[
|G_j(x, \xi) - G_j(y, \xi)| \leq |x - y|~ \phi(\xi),
\]
for all $x,y \in \Xcal$ and for some function $\phi$ satisfying $\Ebb[\phi(\xi)^2] < \infty$.
\end{assumption}
The Lipschitz condition in Assumption \ref{assump:donsker} ensures $G_j(\cdot, \xi)$ is sufficiently smooth for each $\xi$. It is easy to see that Assumption \ref{assump:donsker} holds if the derivative $G_j'(x, \xi)$ is uniformly bounded by a constant for all $\xi$. We use Assumption \ref{assump:donsker} to prove the following lemma. 

\begin{lemma}
\label{lemma:donsker}
Under Assumption \ref{assump:donsker},
\vspace{-0.2cm}
\[
\sup_{x \in \Xcal}\sqrt{L}\left|\hat{G}_{j}(x) - g_j(x)\right| \overset{d}{\longrightarrow} \sup_{x \in \Xcal} \left|\mathbb{G}(x) \right| ~~\text{and}~~ \Ebb\bigg(  \sup_{x \in \Xcal}\left| \hat{G}_{j}(x) - g_j(x) \right| \bigg) = \mathcal{O}(1 / \sqrt{L}),
\]
where $\mathbb{G}(x)$ is a zero mean Gaussian process with covariance function $\Ebb[\mathbb{G}(x)\mathbb{G}(y)] = \Ebb[G(x,\xi)G(y,\xi)] - \Ebb[G(x,\xi)] \Ebb[G(y,\xi)]$.
\end{lemma}

%
\textbf{FKG Inequality:}
The Fortuin-Kasteleyn-Ginibre (FKG) \cite{FKG71} inequality asserts positive correlation between two increasing functions on a finite partially ordered set.
\begin{lemma}
\label{lemma:fkg}
Let $L$ be partially ordered set such that $(L, \vee, \wedge)$ be a finite distributive lattice. Further, let $\mu$ be a probability measure on $L$ such that $\mu(\ell_1\wedge \ell_2) \mu(\ell_1\vee \ell_2) \geq \mu(\ell_1) \mu(\ell_2)$. If both $h_1$ and $h_2$ are increasing functions on $L$ with respect to the partial ordering of $L$, then
\vspace{-0.2cm}
\[
\sum_{\ell \in L} h_1(\ell)h_2(\ell) \mu(\ell) \geq \bigg(\sum_{\ell \in L} h_1(\ell) \mu(\ell) \bigg) \bigg(\sum_{\ell \in L} h_2(\ell) \mu(\ell) \bigg).
\]
\end{lemma}
The following result follows from Lemma \ref{lemma:fkg} and the fact that $f(x_t) - f(x^*) \geq 0$.
\begin{lemma}
\label{lemma:fkg application}
Let $\{\gamma_t\}$, $\{x_t\}$ be as in Algorithm \ref{algo:mcsa} and let $x^* := \underset{x}{\argmin} f(x)$. Then
\vspace{-0.2cm}
\[
\Ebb\bigg[\frac{\sum_{t \in \Bcal} \gamma_t (f(x_t) - f(x^*))}{\sum_{t \in \Bcal} \gamma_t}\bigg] \leq \Ebb\bigg[\sum_{t \in \Bcal} \gamma_t (f(x_t) - f(x^*))\bigg]~\Ebb\bigg[\frac{1}{\sum_{t \in \Bcal} \gamma_t}\bigg]. 
\]
\end{lemma}
\vspace{-0.3cm}
\subsection{Main Result} 
\label{sec:main_result}
\vspace{-0.3cm}
First, we present a sufficient condition under which $\Ebb[|\Bcal|/N] \geq c$. We define, 
\vspace{-0.2cm}
\begin{align}
\label{eq:tau_def}
\tau := D_{\Xcal}^2 +   \frac{M^2}{2}  \sum_{t = s}^N \gamma_t^2,
\end{align}
where $D_\Xcal = \sqrt{\max_{x, z \in \Xcal} B_{\psi}(z, x)}$ denotes the diameter of $\Xcal$.

\begin{theorem}
\label{theorem:welldefined}
Let $\{\gamma_t\}$, $\eta_{j,t}$, $\{x_t\}$ be as in Algorithm \ref{algo:mcsa} and let $x^* := \underset{x}{\argmin} f(x)$. Let
\vspace{-0.3cm}
\begin{align}
\label{eq:welldefined}
\gamma_t = \frac{\sqrt{2}K_1}{\sqrt{N}} ~~\text{and}~~ \eta_{j,t} = \frac{\sqrt{2}K_2}{\sqrt{N}} ~~\text{for all $t$},
\end{align}
for some sufficiently large constant $K_1$ and $K_2$. Then one of the following two conditions hold:
\vspace{-0.2cm}
\begin{enumerate}
\item $P\left(\frac{|\Bcal|}{N} > \frac{1}{2} - \frac{s-1}{N} - \frac{\tau}{2K_1K_2}\right) > \frac{1}{4}$ which implies $\Ebb(|\Bcal| / N) > c$ for some $c \in (0,1]$
\vspace{-0.2cm}
\item $\Bcal \neq \emptyset$ and  $\sum_{t = s}^N  \gamma_t \langle f'(x_t), x_t - x^* \rangle ~\one_{\{t \in \Bcal\}}  \leq 0$ almost surely.
\end{enumerate}
\end{theorem}
\vspace{-0.3cm}
\begin{proof}
We show that if the second condition does not hold then the first condition must hold. Note that if the second condition does not hold,
\vspace{-0.2cm}
\begin{align*}
&\sum_{t = s}^N  \gamma_t \langle f'(x_t), x_t - x^* \rangle ~\one_{\{t \in \Bcal\}}  \geq 0 \qquad
\Rightarrow \qquad \sum_{t = s}^N  \gamma_t ~\Ebb \bigg[ \langle f'(x_t), x_t - x^* \rangle ~\one_{\{t \in \Bcal\}} \biggm| \btxi \bigg] \geq 0,
\end{align*}
where the equality holds if $\Bcal = \emptyset$. Thus it follows from Lemma \ref{lemma:prop9b} with $x = x^*$ that
\vspace{-0.2cm}
\begin{align*}
~\sum_{t = s}^N \sum_{j=1}^m \gamma_t~\Ebb \bigg[\big(g_j(x_t) -  g_j(x^*) \big) ~\one_{\{t \in \Ncal_j\}} \biggm| \btxi \bigg] \leq ~ \Ebb \bigg[ B_{\psi}(x^*, x_{s}) \bigg| \btxi \bigg]  ~+~ \frac{M^2}{2}  \sum_{t = s}^N \gamma_t^2.
\end{align*}
Moreover, since $g_j(x^*) \leq 0$, $B_{\psi}(x^*, x_s) \leq D_{\Xcal}^2$, we have
\vspace{-0.1cm}
\begin{align}
\label{eq:lemma41}
\sum_{t = s}^N \sum_{j=1}^m \gamma_t \Ebb \left[ g_j(x_t) ~\one_{\{t \in \Ncal_j\}} \bigg| \btxi \right] \leq \tau 
\end{align}
where $\tau$ is defined in \eqref{eq:tau_def}. Note that, for $t \in \Ncal_j$, we have $\hat{G}_{j}(x_t) > \eta_{j,t}$. Thus we have,
\vspace{-0.1cm} 
\begin{align}
\label{eq:lemma42}
 \sum_{t = s}^N \sum_{j=1}^m \gamma_t \bigg( \hat{G}_{j}(x_t) \one_{\{t \in \Ncal_j\}} - \Ebb \left[ g_j(x_t) ~\one_{\{t \in \Ncal_j\}} \bigg| \btxi \right] \bigg) > \sum_{t = s}^N \sum_{j=1}^m \gamma_t \eta_{j,t} \one_{\{t \in \Ncal_j\}} - \tau.
\end{align}
Now consider the following two sets:
\vspace{-0.2cm}
\begin{align}
A_1 &:= \bigg\{\sum_{t = s}^N \sum_{j=1}^m \gamma_t (\hat{G}_{j}(x_t) - g_j(x_t)) \one_{\{t \in \Ncal_j\}} \leq \frac{1}{2} \sum_{t = s}^N \sum_{j=1}^m \gamma_t \eta_{j,t} \one_{\{t \in \Ncal_j\}} \bigg\}; \\
A_2 &:= \bigg\{\sum_{t = s}^N \sum_{j=1}^m \gamma_t \left( g_j(x_t)\one_{\{t \in \Ncal_j\}} -  \Ebb \left[ g_j(x_t) ~\one_{\{t \in \Ncal_j\}} \bigg| \btxi \right]\right) \leq \frac{K_1K_2 - \tau}{2} \bigg\}.
\end{align}

On the set $A_1 \cap A_2$, we have
\vspace{-0.2cm}
\begin{align}\label{eq:intersection}
 \sum_{t = s}^N \sum_{j=1}^m \gamma_t \big( \hat{G}_{j}(x_t) \one_{\{t \in \Ncal_j\}} - \Ebb \big[ g_j(x_t) ~\one_{\{t \in \Ncal_j\}} \big| \btxi \big] \big) \leq \frac{1}{2}\sum_{t = s}^N \sum_{j=1}^m \gamma_t \eta_{j,t} \one_{\{t \in \Ncal_j\}} + \frac{K_1K_2 - \tau}{2}.
\end{align}
By combining \eqref{eq:welldefined}, \eqref{eq:lemma42} and \eqref{eq:intersection}, we get
\vspace{-0.2cm}
\begin{align*}
\frac{2K_1K_2(N-s+1 - |\Bcal|)}{N} \leq \sum_{t = s}^N \sum_{j=1}^m \gamma_t \eta_{j,t} \one_{\{t \in \Ncal_j\}} < K_1K_2 + \tau.
\end{align*}
This implies on the set $A_1 \cap A_2$, we have $\frac{|\Bcal|}{N} > \frac{1}{2} - \frac{s-1}{N} - \frac{\tau}{2K_1K_2}$. Thus, we get,
\vspace{-0.2cm}
\begin{align*}
P\left(\frac{|\Bcal|}{N} > \frac{1}{2} - \frac{s-1}{N} - \frac{\tau}{2K_1K_2}\right) & \geq P(A_1 \cap A_2) \geq 1 - P(A_1^c) - P(A_2^c).
\end{align*}
Next, we derive upper bounds for $P(A_1^c)$ and $P(A_2^c)$. From straightforward calculations, it follows that
\vspace{-0.2cm}
\begin{align}
\label{eq:firstbound}
P(A_1^c) & \leq P\bigg(\sum_{t = s}^N \sum_{j=1}^m \left|\gamma_t (\hat{G}_{j}(x_t) - g_j(x_t)) \one_{\{t \in \Ncal_j\}}\right| > \sum_{t=s}^N \sum_{j=1}^m \frac{K_1K_2}{N} \bigg) \nonumber \\
& \leq \sum_{j=1}^m P\bigg( \sup_{x \in \Xcal} \left| \sqrt{L} (\hat{G}_{j}(x) - g_j(x))\right| >   \sqrt{\frac{LK_2^2}{2N}} \bigg) < \frac{1}{2}.
\end{align} 
where the last inequality follows from Lemma \ref{lemma:donsker} for sufficiently large $K_2$, since we have chosen $L = \Omega(N)$. Now using Lemma \ref{lemma:conc} and choosing $\lambda = (K_1K_2 - \tau)/\sqrt{4 \sum_{t=s}^N \gamma_t^2 \sigma^2}$, we have 
\vspace{-0.1cm}
\begin{align}
\label{eq:secondbound}
P(A_2^c) & \leq \exp \left( -\frac{N (K_1K_2 - \tau)^2}{12 \sigma^2 (N-s+1) K_1^2} \right) < \frac{1}{4},
\end{align} 
where the last inequality follows from choosing suitably large $K_1, K_2$. Thus, using \eqref{eq:firstbound} and \eqref{eq:secondbound} we get
\vspace{-0.4cm}
\begin{align*}
P\left(\frac{|\Bcal|}{N} > \frac{1}{2} - \frac{s-1}{N} - \frac{\tau}{2K_1K_2}\right) > \frac{1}{4}.
\end{align*}
\vspace{-0.2cm}
\end{proof}
\begin{theorem}
\label{thm:main1}
Let $\{\gamma_t\}$, $\eta_{j,t}$, $\{x_t\}$, $x^*$ be as in Theorem \ref{theorem:welldefined}. Then under Assumptions \ref{assump:norm} and \ref{assump:donsker}, we have
\vspace{-0.1cm}
\begin{align}
\label{eq:rate_eq}
\Ebb[f(\hat{\xb}) - f(x^*)] = \mathcal{O}\bigg(\frac{1}{\sqrt{N}}\bigg) \;\;\text{ and } \;\;\Ebb[ g_j(\hat{\xb})] = \mathcal{O}\bigg(\frac{1}{\sqrt{L}} + \frac{1}{\sqrt{N}}\bigg) \;\; \forall j. 
\end{align}
\end{theorem}

\begin{proof}
First observe that the second condition of Theorem \ref{theorem:welldefined} implies that our Algorithm has already converged. Specifically, if $\Bcal \neq \emptyset$ our algorithm is well-defined and we have,
\vspace{-0.2cm}
\begin{align}
\label{eq:rightbound1}
0 &\geq \sum_{t = s}^N  \gamma_t \langle f'(x_t), x_t - x^* \rangle ~\one_{\{t \in \Bcal\}} \geq  \sum_{t = s}^N \gamma_t (f(x_t) - f(x^*)) ~\one_{\{t \in \Bcal\}} \nonumber \\
& \geq \bigg( \sum_{t = s}^N  \gamma_t \one_{\{t \in \Bcal\}} \bigg) (f(\hat{\xb}) - f(x^*))
\end{align}
where we have used the successively used the convexity of $f$. Since $\Bcal \neq \emptyset$, we get $f(\hat{\xb}) - f(x^*) \leq 0$ and hence the algorithm has already converged, i.e.\ $\hat{\xb} = x^*$. In this case, we have $\Ebb[f(\hat{\xb}) - f(x^*)] = 0$ and $\Ebb[ g_j(\hat{\xb})] = \Ebb[ g_j(x^*)] \leq 0$. Thus, either our algorithm has already converged or the first condition of Theorem \ref{theorem:welldefined} holds. From the first condition of Theorem \ref{theorem:welldefined} we have $\Ebb(|\Bcal| / N) > c$ for sufficiently large $N$ and for some $c > 0$. Now using convexity of $g_j$ and the fact that $\hat{G}_{j}(x_t) \leq \eta_{j,t}$ for $t \in \Bcal$, we have
\vspace{-0.2cm}
\begin{align*}
\Ebb( g_j(\hat{\xb}) ) & \leq \Ebb \left[ \frac{\sum_{t=s}^N \gamma_t g_j(x_t) \one_{\{t \in \Bcal\}}}{ \sum_{t = s}^N \gamma_t \one_{\{t \in \Bcal\}}}\right] \\
&= \Ebb \left[ \frac{\sum_{t=s}^N \gamma_t (g_j(x_t) - \hat{G}_{j}(x_t))\one_{\{t \in \Bcal\}}}{ \sum_{t = s}^N \gamma_t \one_{\{t \in \Bcal\}}}\right] + \Ebb \left[ \frac{\sum_{t=s}^N \gamma_t \hat{G}_{j}(x_t)\one_{\{t \in \Bcal\}}}{ \sum_{t = s}^N \gamma_t \one_{\{t \in \Bcal\}}}\right] \\
& \leq \Ebb \left[ \sup_{x} (g_j(x) - \hat{G}_{j}(x))\right] +  \frac{\sqrt{2}K_2}{\sqrt{N}} = \mathcal{O}\bigg(\frac{1}{\sqrt{L}} + \frac{1}{\sqrt{N}}\bigg),
\end{align*}
where the we have used Lemma \ref{lemma:donsker} to get the last equality. Similarly, using convexity of $f$
\begin{align}
\label{eq:first_bound}
\Ebb \left[ f(\hat{\xb}) - f(x^*)\right] &\leq \Ebb \left[ \frac{\sum_{t=s}^N \gamma_t (f(x_t) - f(x^*))\one_{\{t \in \Bcal\}}}{ \sum_{t = s}^N \gamma_t \one_{\{t \in \Bcal\}}}\right] \nonumber \\
& \leq \Ebb \left[ \sum_{t=s}^N \gamma_t (f(x_t) - f(x^*))\one_{\{t \in \Bcal\}}\right] \Ebb\left[\frac{1}{ \sum_{t = s}^N \gamma_t \one_{\{t \in \Bcal\}}}\right] 
\end{align}
where the last inequality follows from Lemma \ref{lemma:fkg application}. Note that applying Jensen's inequality for the concave function $1/x$ we have,
\vspace{-0.2cm}
\begin{align}
\label{eq:jensen}
\Ebb\left[\frac{1}{ \sum_{t = s}^N \gamma_t \one_{\{t \in \Bcal\}}}\right] \leq \frac{1}{\Ebb \left[ \sum_{t = s}^N \gamma_t \one_{\{t \in \Bcal\}} \right] } \leq \frac{\sqrt{N}}{\sqrt{2}K_1 N\Ebb\left[ \frac{| \Bcal|}{N}\right]} \leq \frac{1}{(\sqrt{2}cK_1)\sqrt{N}}.
\end{align}
where we have used the definition of $\gamma_t$ and the fact that $\Ebb(|\Bcal| / N) > c$. Moreover, using the fact that both $x_t$ and $\one_{\{t \in \Bcal\}}$ do not depend on $\xi_{0,t}$ we have, 
\vspace{-0.2cm}
\begin{align}
\label{eq:first_term}
\Ebb &\bigg[ \sum_{t=s}^N  \gamma_t (f(x_t) - f(x^*))\one_{\{t \in \Bcal\}}\bigg] = 
\sum_{t=s}^N \gamma_t \Ebb \left[  (F(x_t, \xi_{0,t}) - F(x^*, \xi_{0,t}))\one_{\{t \in \Bcal\}}\right]
\nonumber \\
& \leq \sum_{t=s}^N \gamma_t \Ebb \left[\langle F'(x_t, \xi_{0,t}), x_t - x^* \rangle\one_{\{t \in \Bcal\}}\right]
\nonumber \\
& \leq \Ebb(B_{\psi}(x^*, x_{s})) - \sum_{t = s}^N \sum_{j=1}^m  \Ebb \left[\gamma_t \big(G_j(x_t, \xi_{j,t}) -  G_j(x^*, \xi_{j,t})\big) \one_{\{t \in \Ncal_j\}}\right] + \frac{M^2}{2}  \sum_{t = s}^N \gamma_t^2 \nonumber \\
& \leq \tau - \sum_{t = s}^N \sum_{j=1}^m \gamma_t~\Ebb\big[ g_j(x_t) ~\one_{\{t \in \Ncal_j\}} \big] 
\end{align}
where the first inequality follows from the convexity of $F(\cdot, \xi_{j,t})$, the second inequality follows from Lemma \ref{lemma:prop9} with $x = x^*$ and the definition of $M$ given in Lemma \ref{lemma:prop9b} and the third inequality follows the definition of $\tau$ from \eqref{eq:tau_def} and the fact that $g_j(x^*) \leq 0$. By plugging in \eqref{eq:jensen} and \eqref{eq:first_term} into \eqref{eq:first_bound} we get,
\vspace{-0.2cm}
\begin{align*}
\Ebb \left[ f(\hat{\xb}) - f(x^*)\right] &\leq \frac{1}{(\sqrt{2}cK_1)\sqrt{N}} \bigg(\tau - \sum_{t = s}^N \sum_{j=1}^m \gamma_t~\Ebb\big[ g_j(x_t) ~\one_{\{t \in \Ncal_j\}} \big]  \bigg)
\end{align*}
Finally,
\vspace{-0.2cm}
\begin{align*}
\Ebb\left[ g_j(x_t) ~\one_{\{t \in \Ncal_j\}} \right]  &= \Ebb\left[ (g_j(x_t) - \hat{G}_{j}(x_t))~\one_{\{t \in \Ncal_j\}} \right] + \Ebb\left[ \hat{G}_{j}(x_t)~\one_{\{t \in \Ncal_j\}} \right] \\
& \geq \Ebb\left[ (g_j(x_t) - \hat{G}_{j}(x_t))~\one_{\{t \in \Ncal_j\}} \right] + \eta_{j,t} \\
& \geq -\Ebb\left[ \sup_{x} \left|\hat{G}_{j}(x) - g_j(x)\right| \right] + \frac{\sqrt{2}K_2}{\sqrt{N}} \geq -\frac{C_2}{\sqrt{L}}~~\text{for some $C_2 > 0$,}
\end{align*}
where the last inequality follows from the second part of Lemma \ref{lemma:donsker} and the fact that $K_2 > 0$.
Thus, we have
\vspace{-0.4cm}
\begin{align*}
\Ebb \left[ f(\hat{\xb}) - f(x^*)\right] &\leq \frac{1}{(\sqrt{2}cK_1)\sqrt{N}}  \bigg(\tau + \sum_{t = s}^N \sum_{j=1}^m  \frac{C_2\gamma_t}{\sqrt{L}} \bigg) =  \mathcal{O}\bigg( \frac{1}{\sqrt{N}} \bigg),
\end{align*}
where the last equality follows from the fact that $\gamma_t = \mathcal{O}(1/\sqrt{N})$.
\end{proof}

\vspace{-0.5cm}
\section{Experiments}
\label{sec:experiments}
\vspace{-0.3cm}
In this section, we describe simulated experiments
to showcase the efficacy of our algorithm. Throughout this section, we compare our algorithm to the state-of-the-art algorithm in online convex optimization with stochastic constraints \cite{yu2017online}. We focus our experiments on a real-world problem which motivated our work. 
\vspace{-0.2cm}
\subsection{Personalized Parameter Selection in Large-Scale Social Network}
\vspace{-0.3cm}
Most social networks do extensive experimentation to identify which parameter gives rise to the biggest online metric gains. However, in many cases choosing a single global parameter is not very prudent. That being said, experimenting to identify member level parameter is an extremely challenging problem. A good middle pathway lies in identifying cohorts of members and estimating the optimal parameter for each cohort. \cite{tu2019prophet} tries to solve this problem by framing this problem as \eqref{eq:main_problem}. 

Let us focus on the minimum gap parameter in the ranking ads problem. Let us assume that the parameter can take $K$ possible  values and there are $M + 1$ metrics we are interested in. One primary metric (revenue) and $M$ guardrail metrics (ads click-through-rate, organic click-through-rate, etc). Specifically, we can estimate the causal effect when treatment $k \in K$ is applied to a cohort for each of these $M+1$ metrics. This effect is a random variable which is usually distributed as a Gaussian with some mean and variance. Our aim is to identify the optimal allocation $x^*$ which maximizes the expected revenue. Formally, let $x_{ik}$ denote the probability of of assigning the $k$-th treatment to the $i$-th cohort and let $U_{ik}^j$ denote the metric lift in the $j$-th metric when $k$-th treatment is assignment to $i$-th cohort. 
\vspace{-0.1cm}
Thus we want to solve,
\begin{equation}
\label{eq:example_main}
\begin{aligned}
& \underset{x}{\text{Maximize}} &  \sum_{ik} x_{ik}U_{ik}^0\\
& \text{subject to}
&  \sum_{ik}x_{ik}U_{ik}^j \leq c_j &\qquad \text{ for } j = 1, \ldots, M. \\
&& 0 \leq x_{ik} \leq 1 &\qquad \text{ for all } i,k \qquad \text{ and } \qquad \sum_{k} x_{ik} = 1 \;\;\; \text{ for all } i.
\end{aligned}
\end{equation}
\vspace{-0.4cm}

Note that each $U_{ik}^j$ is a random variable distributed usually as a Gaussian distribution with some mean and variance. Hence this problem can be translated to the following:
\vspace{-0.1cm}
\begin{equation}
\label{eq:example_problem}
\begin{aligned}
& \underset{x}{\text{Maximize}} & & f(x) := \Ebb_{ \xi_0} \left(x^T\xi_0\right)\\
\vspace{-0.5cm}
& \text{subject to}
& & g_j(x) :=  \Ebb_{ \xi_j} \left(x^T \xi_j - c_j\right) \leq 0\qquad \text{ for } j = 1, \ldots, M, \qquad \text{ and } \qquad x \in \Xcal
\end{aligned}
\end{equation}
where $\Xcal$ are the set of non-stochastic constraints. For simplicity in our simulation setting we choose, $\Xcal = [0,1]^d$, and $c_j = 0$ for all $j$. Further, we assume that each $\xi_j \sim N_d(\mu, \Sigma)$ for $j = 0, \ldots, M$. Note that, this also satisfies Assumptions \ref{assump:norm} and \ref{assump:donsker}. All parameter settings are pushed to the supplementary for brevity. 

We run the algorithm using different configurations of the mean and variance. 
Although, we fix $\btxi$ in our algorithm to get $\hat{G}_{j}$ in \eqref{eq:Gfunc1}, it can very easily be changed to an online version, by choosing
$\hat{G}_{j}(x_t) := \frac{1}{t} \sum_{\ell = 1}^t G_j(x_t, \tilde{\xi}_{j,\ell})$.
We call this the MCSA-online algorithm and show the results to ensure a fair comparison with Yu et. al. (2017) \cite{yu2017online}. The results are shown in Figure \ref{fig:main}. 
\vspace{-0.2cm}
\begin{figure*}[!th] 
\centering
\includegraphics[width=\linewidth]{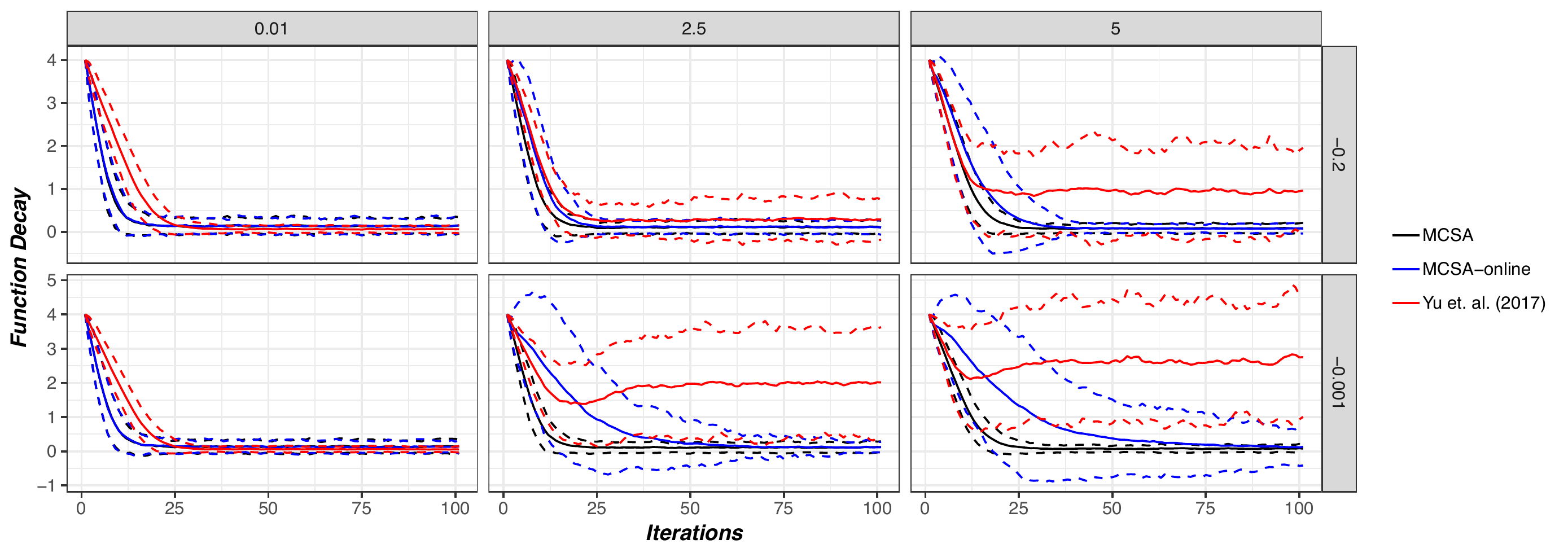}
\caption{\label{fig:main} This shows the function decay vs iterations for six different simulation settings corresponding to $\mu \in \{-0.2, -0.001\}$ and $\sigma^2 \in \{0.01, 2.5, 5\}$. The top and the bottom rows correspond to an easy and a more difficult feasibility domain respectively, while the columns depict an increase in variance we move towards the right. The dotted lines denote the 95\% confidence interval using 100 repeats of each configuration.}
\end{figure*}

From Figure \ref{fig:main}, we observe that in a high signal-to-noise ratio setting all algorithms work well. For the settings with low signal-to-noise ratios, MCSA and MCSA-online outperform the state-of-the-art. For a harder problem setting, a more clear evidence can be seen that our algorithm performs better than \cite{yu2017online}. We have also tried a similar exercise with different dimensions $d$ and different number of constraints $M$ and in each case, we observed a similar result.

\vspace{-0.3cm}
\section{Discussion}
\label{sec:discussion}
\vspace{-0.3cm}
In this paper, we have described a solution to a stochastic optimization problem specifically containing multiple expectation constraints. We introduce the MCSA algorithm, prove its optimal convergence by a careful consideration of the dependent structure. This work also plugs in several gaps in the proof of \cite{lan2016algorithms}. We compare an online version of our algorithm to the state-of-the-art and empirically show some instances where it performs better. As a future step, we are interested in extending the algorithm to more general problems, especially the cases where the expectation constraints are replaced by probability. Being able to solve such general problems will have a high impact in many research and scientific domains. 

\section*{Acknowledgement}
We thank Abhishek Chakrabortty for many helpful discussions on empirical processes and Donsker's theorem.

\bibliographystyle{abbrv}
\bibliography{stocProg} 
\vfill
\pagebreak
\section{Supplementary Material}
\label{sec:appendix}

\subsection{Proof of Supporting Lemmas}

\subsubsection{Proof of Lemma \ref{lemma:breg}}
We refer to \cite{nemirovski2009robust} for a proof. 

\subsection{Proof of Lemma \ref{lemma:prop9}}
The proof of this result is almost identical to the proof of Proposition 9 in \cite{lan2016algorithms}. We present this here for completeness. For any $s \leq t \leq N$, from Lemma \ref{lemma:breg} we have,
\begin{align}
\label{eq:lemma21}
B_{\psi}(x, x_{t+1}) \leq B_{\psi}(x, x_t) + \gamma_t \langle h_t, x - x_t \rangle + \frac{\gamma_t^2}{2} \norm{h_t}_{\psi^*}^2.
\end{align} 
Note that, if $t \in \Bcal$, $\langle h_t, x_t - x \rangle = \langle F'(x_t, \xi_{0,t}), x_t - x \rangle$ and if $t \in \Ncal_j$, then due to convexity of $G_j$, 
\begin{align*}
\langle h_t, x_t - x \rangle &= \langle G_j'(x_t, \xi_{j,t}), x_t - x \rangle \geq G_j(x_t, \xi_{j,t}) -  G_j(x, \xi_{j,t}).
\end{align*}
Moreover, we have
\begin{align}
\label{eq:fourth term}
\norm{h_t}_{\psi^*}^2  = \norm{F'(x_t, \xi_{0,t})}_{\psi^*}^2  \one_{\{t \in \Bcal\}} + \sum_{j=1}^m \norm{G'(x_t, \xi_{j,t})}_{\psi^*}^2 \one_{\{t \in \Ncal_j\}}  
\end{align}
Thus we get,
\begin{align}
\label{eq:fundamental inequality}
B_{\psi}(x, x_{t+1}) &\leq B_{\psi}(x, x_t) - \gamma_t \langle F'(x_t, \xi_{0,t}), x_t - x \rangle ~\one_{\{t \in \Bcal\}} \\
&\qquad - \sum_{j=1}^m\gamma_t \big(G_j(x_t, \xi_{j,t}) -  G_j(x, \xi_{j,t})\big)~\one_{\{t \in \Ncal_j\}} \\
& \qquad + \frac{\gamma_t^2}{2}\left(\norm{F'(x_t, \xi_{0,t})}_{\psi^*}^2  \one_{\{t \in \Bcal\}} + \sum_{j=1}^m \norm{G'(x_t, \xi_{j,t})}_{\psi^*}^2 \one_{\{t \in \Ncal_j\}} \right) \qquad \text{a.s.}
\end{align}
Adding up from $t = s$ to $N$ we get,
\begin{align*}
B_{\psi}(x, x_{N+1}) &\leq B_{\psi}(x, x_{s}) - \sum_{t = s}^N \gamma_t \langle F'(x_t, \xi_{0,t}), x_t - x \rangle ~\one_{\{t \in \Bcal\}}\\
& - \sum_{t = s}^N\sum_{j=1}^m\gamma_t \big(G_j(x_t, \xi_{j,t}) -  G_j(x, \xi_{j,t})\big)~\one_{\{t \in \Ncal_j\}}\\
&+ \sum_{t=s}^N \frac{\gamma_t^2}{2}\left(\norm{F'(x_t, \xi_{0,t})}_{\psi^*}^2  \one_{\{t \in \Bcal\}} + \sum_{j=1}^m \norm{G'(x_t, \xi_{j,t})}_{\psi^*}^2 \one_{\{t \in \Ncal_j\}} \right)  \qquad \text{a.s.}
\end{align*}
Rearranging the terms and using $ B_{\psi}(x, x_{N+1}) \geq 0$ we arrive at the required inequality, hence proving the result.
\qed

\subsubsection{Proof of Lemma \ref{lemma:prop9b}}
By taking a conditional expectation given $x_t$ and $\btxi$ on both sides of \eqref{eq:fundamental inequality} and then adding up from $t = s$ to $N$ we get,
\begin{align*}
\sum_{t = s}^N \Ebb &\left[\gamma_t \langle F'(x_t, \xi_{0,t}), x_t - x \rangle~\one_{\{t \in \Bcal\}} \mid x_t, \btxi \right] +  \nonumber\\
& \sum_{j=1}^m \sum_{t = s}^N \Ebb\left[ \gamma_t \big(G_j(x_t, \xi_{j,t}) -  G_j(x, \xi_{j,t})\big)~\one_{\{t \in \Ncal_j\}} \mid x_t, \btxi \right] \nonumber \\
& \leq \Ebb [B_{\psi}(x, x_{s}) \bigm| x_t, \btxi ] + \sum_{t = s}^N \frac{\gamma_t^2}{2}  \Ebb \left[ \norm{F'(x_t, \xi_{0,t})}_{\psi^*}^2 ~ \one_{\{t \in \Bcal\}}  \bigm| x_t,  \btxi \right] + \nonumber \\
 & \qquad \sum_{t = s}^N \sum_{j=1}^m  \frac{\gamma_t^2}{2}  \Ebb \left[ \norm{G'(x_t, \xi_{j,t})}_{\psi^*}^2~ \one_{\{t \in \Ncal_j\}} \bigm| x_t , \btxi\right] 
\end{align*}
Note that given $x_t$, the random variables $\one_{\{t \in \Bcal\}}$ and $\one_{\{t \in \Ncal_j\}}$ (depending on 
$\boldsymbol{\tilde{\xi}} :=  \{ \{ \tilde{\xi}_{j,\ell}\}_{\ell = 1}^L\}_{j=1}^m$) are independent of $\bxi_t$. Hence we get,
\begin{align}
\label{eq:second term}
\Ebb \bigg[ \langle F'(x_t, \xi_{0,t}), x_t - x \rangle ~\one_{\{t \in \Bcal\}} \biggm| x_t, \btxi \bigg] = &~ \Ebb \big[ \langle F'(x_t, \xi_{0,t}), x_t - x \rangle \bigm| x_t \big] ~ \one_{\{t \in \Bcal\}} \nonumber \\
=&~ \langle f'(x_t), x_t - x \rangle \one_{\{t \in \Bcal\}},
\end{align}
and
\begin{align}
\label{eq:third term}
\Ebb \bigg[ \big(G_j(x_t, \xi_{j,t}) -  G_j(x, \xi_{j,t}) \big)~\one_{\{t \in \Ncal_j\}} \biggm| x_t, \btxi \bigg]= &~ \Ebb \big[G_j(x_t, \xi_{j,t}) -  G_j(x, \xi_{j,t}) \bigm| x_t \big] \one_{\{t \in \Ncal_j\}} \nonumber \\
=&~ (g(x_t) - g(x))  ~ \one_{\{t \in \Ncal_j\}}.
\end{align}
Similarly, we have
\begin{align}
\label{eq:fourth term}
&\sum_{t = s}^N \frac{\gamma_t^2}{2}  \Ebb \left[ \norm{F'(x_t, \xi_{0,t})}_{\psi^*}^2 ~ \one_{\{t \in \Bcal\}} \bigm| x_t,   \btxi \right] +  \sum_{t = s}^N \sum_{j=1}^m  \frac{\gamma_t^2}{2}  \Ebb \left[ \norm{G'(x_t, \xi_{j,t})}_{\psi^*}^2~ \one_{\{t \in \Ncal_j\}} \bigm| x_t , \btxi\right]  \nonumber \\
& = \sum_{t = s}^N \frac{\gamma_t^2}{2}  \Ebb \big[ \norm{F'(x_t, \xi_{0,t})}_{\psi^*}^2 \bigm| x_t \big] ~ \one_{\{t \in \Bcal\}} + \sum_{t = s}^N \sum_{j=1}^m  \frac{\gamma_t^2}{2}  \Ebb \big[ \norm{G'(x_t, \xi_{j,t})}_{\psi^*}^2 \bigm| x_t \big] ~ \one_{\{t \in \Ncal_j\}}  \nonumber \\
& \leq \sum_{t = s}^N \frac{\gamma_t^2}{2} M_F^2 ~ \one_{\{t \in \Bcal\}} + \sum_{t = s}^N \sum_{j=0}^m \frac{\gamma_t^2}{2} M_{G_j}^2  \one_{\{t \in \Ncal_j\}}  \leq M^2 \sum_{t = s}^N \frac{\gamma_t^2}{2}
\end{align}
where the first inequality follows from Assumption \ref{assump:norm} and the second inequality follows from the definition of $M$. By combining \eqref{eq:second term}, \eqref{eq:third term} and \eqref{eq:fourth term} and by taking a conditional expectation over $x_t$ given $\btxi$, we obtain the desired inequality, hence proving the result.

\subsubsection{Proof of Lemma \ref{lemma:conc}}
The proof follows from Lemma 7 of \cite{lan2016algorithms}. 
Note that $\zeta_t$ conditional on $\btxi$ is a deterministic function of $\xi_1, \ldots, \xi_{t-1}$, where $\xi_t = (\xi_{1,t}, \ldots, \xi_{m,t})$ is the sample drawn at $t$-th iteration. Now, 
$\Ebb(\zeta_t | \btxi) = 0$ and from Assumption \ref{assump:lighttail}, $\Ebb(\exp(\zeta_t^2 / \gamma_t^2 \sigma^2) | \btxi) \leq \exp(1)$. 
Thus, now from the result of Lemma 7 in \cite{nemirovski2009robust} we have for all $\lambda > 0$,
\begin{align*}
P\left(\sum_{t=s}^{N} \zeta_t > \lambda \sqrt{\sum_{t=s}^N \gamma_t^2 \sigma^2} \;\;\Bigg| \;\;\btxi \right) \leq \exp\bigg(-\frac{\lambda^2}{3}\bigg). 
\end{align*}
\qed

\subsubsection{Proof of Lemma \ref{lemma:donsker}}
Fix $j \in \{1,\ldots,m\}$. We begin by defining the \emph{bracketing number} $N_{[]}(\epsilon, \Fcal_j, L_2(P))$ of the function class $\Fcal_j = \{G_j(x, \cdot) : x \in \Xcal \}$ with respect to the $L_2(P)$ norm defined as $||G_j(x, \cdot)||_{L_2(P)}^2 = \Ebb[G_j(x, \xi)^2]$. The set of all functions in $\Fcal$ satisfying $||G_j(x, \xi) - G_j(y, \xi)||_{L_2(P)} < \epsilon$ is an \emph{$\epsilon$-bracket} with boundaries $G_j(x, \xi)$ and $G_j(y, \xi)$. Then the bracketing number $N_{[]}(\epsilon, \Fcal_j, L_2(P))$ is defined as the minimum number of $\epsilon$-bracket required to cover $\Fcal_j$.

Next, we formally define the Donsker property for $\Fcal_j$ and its relation with the bracketing number. The function class $\Fcal_j$ is said to be $P$-Donsker if $\hat{G}(x)$ converges in distribution to a Gaussian random element (called Brownian bridge) in the space $\ell^{\infty}(\Xcal)$ of all bounded functions $h: \Xcal \rightarrow \mathbb{R}$, equipped with the sup norm $||h||_{\sup} = \sup_{x} |h(x)|$. A sufficient condition for a function class $\Fcal$ to be Donsker $P$-Donsker is given by 
\begin{align}
\label{eq:suff cond}
N_{[]}(\epsilon, \Fcal, L_r(P) = \mathcal{O}\bigg( \frac{1}{\epsilon^d} \bigg),~~\text{for some $d > 0$}.
\end{align}
We note that \eqref{eq:suff cond} holds for $\Fcal = \Fcal_j$ with $d = dimension(\Xcal)$. This follows from Assumption \ref{assump:donsker}, Theorem 2.7.11 of \cite{VaartWellner96} and the fact that $\Xcal$ is a compact set. Therefore, the first part of Lemma \ref{lemma:donsker} follows from the continuous mapping theorem, since $h(x, \xi) = \sup_{x \in \Xcal}|G_j(x, \xi) - g(x)|$ is a continuous function for each $\xi$.

The second part of Lemma \ref{lemma:donsker} follows from \eqref{eq:suff cond} and Theorem 2.14.2 of \cite{VaartWellner96}. \qed

\subsubsection{Proof of Lemma \ref{lemma:fkg}}
We refer to \cite{FKG71} for a proof.

\subsubsection{Proof of Lemma \ref{lemma:fkg application}}
Let $T = \{s, \ldots, N\}$ and let $\Omega(T)$ be the set of all subsets of $T$. We define a partial ordering on $\Omega(T)$ of as follows: $B_1 \preceq B_2$ if and only if $B_1 = \{t_1,\ldots,t_{k_1}\}$ and $B_2 = \{t_1,\ldots,t_{k_2}\}$ where $s \leq t_1 < t_2 < \cdots < t_{k_1} < \cdots < t_{k_2} \leq N$. We define a probability distribution on $\Omega(T)$ as $\mu(A) : = P(\Bcal = A)$. Then it is easy to verify that $(\Omega(T), \vee, \wedge)$ is a finite distributive lattice for
\[
A\vee B := \argmax \{\mu(A), \mu(B)\} ~~\text{and}~~ A\wedge B := \argmin \{\mu(A), \mu(B)\}.
\]
Therefore, we have $P(A\vee B)~P(A\wedge B) = P(A) P(B)$. Finally, note that both $h_1(B) := \Ebb [ \sum_{t \in \Bcal} \gamma_t (f(x_t) - f(x^*)) \mid \Bcal = B]$ and $h_2(B) := -\Ebb [ 1 /(\sum_{t \in \Bcal} \gamma_t) \mid \Bcal = B]$ are increasing functions in $\Omega(T)$ as $f(x) - f(x^*) \geq 0$ for all $x \in \Xcal$ and $\gamma_t > 0$ for all $t$. Thus Lemma \ref{lemma:fkg} follows directly from Lemma \ref{lemma:fkg}. \qed

\subsection{Experimental Setup}

To initialize, we choose $x_1 = 0.5 \times \one_d$, a $d$-dimensional vector where each co-ordinate is 0.5 
and we fix $d = 100$. Moreover we choose $\{\eta_j\}_t = M^2/\sqrt{N}$ and $\{\gamma\}_t  = D_\Xcal/(M\sqrt{N})$. We choose $\psi(x) = x^Tx$ and $M = 10$. We first convert this into a minimization problem and we pick, $\mu_0 = -0.8\times\one_d$ and $\Sigma_0 = \Ib_{d\times d}$.

\subsubsection{Feasible Configuration:}
We begin with a feasible configuration. We choose, $\mu_j = \mu \times \one_d$,  $\Sigma_j = \sigma^2 \Ib_{d\times d}$ for $j = 1, \ldots, m$. We run both variants of the MCSA algorithm along with the state-of-the-art \cite{yu2017online} for difference choices of $(\mu,\sigma^2)$.  Specifically, we choose $\mu = \{-0.2, -0.001\}$ and $\sigma^2 = \{0.01, 2.5, 5\}$. Note that $\mu = -0.2$ it is a relatively easier problem since the feasible region is easier to obtain as compared to $\mu = -0.001$.

In all of these settings, the optimal solution is $f(x^*) = 80$. Each of these six setting are repeated 100 times and we plot the mean function decay $f(x_t) -  f(x^*)$ across the iterations. We also add a 95\% confidence interval from the repeated experiments. Throughout we are considering the corresponding minimization problem corresponding to \eqref{eq:example_problem}.

\subsubsection{Infeasible Configuration:}
Note that as long as $\mu < 0$, the problem in \eqref{eq:example_problem} is a feasible problem. To test infeasibility, we choose $\mu = 0.2 \times \one_d$ and $\sigma^2 = 1$. Furthermore we increase the number of iterations $N = 10000$. In this cases, both the algorithms failed to converge as expected. The plot is for MCSA shown in Figure \ref{fig:infeasible}.

\begin{figure} [ht!]
\centering
\includegraphics[scale=0.35]{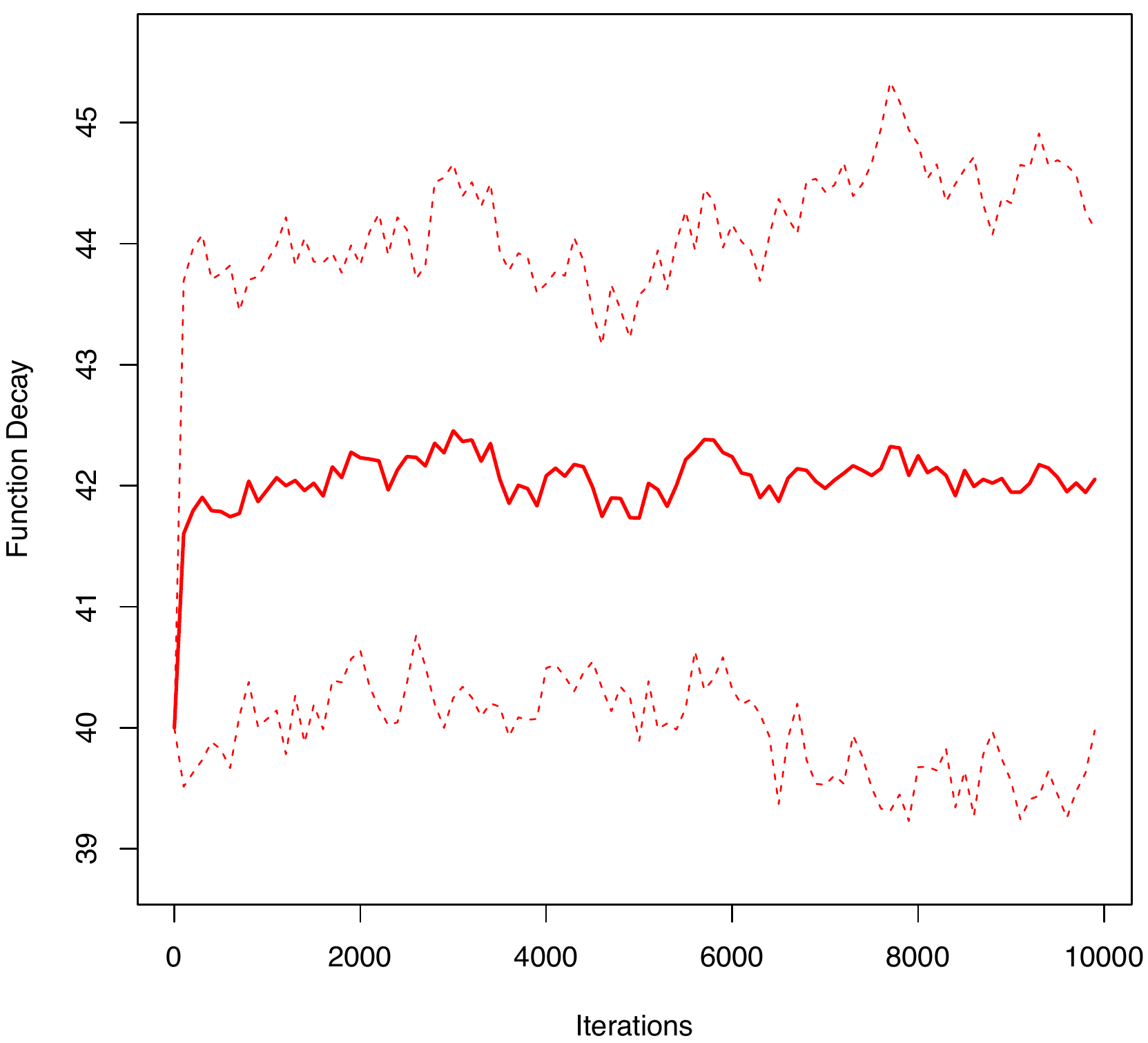}
\caption{Function decay for an infeasible problem.}
\label{fig:infeasible}
\vspace{-0.3cm}
\end{figure}

\end{document}